\documentclass[11pt,reqno,letter]{amsart}

\usepackage{amsmath,amsthm,amssymb}
\usepackage{enumerate}

\newtheorem{theorem}{Theorem}[section]
\theoremstyle{plain}
\newtheorem{corollary}[theorem]{Corollary}
\newtheorem{example}[theorem]{Example}

\newtheorem{proposition}[theorem]{Proposition}
\newtheorem{remark}[theorem]{Remark}

\numberwithin{equation}{section}

\newcommand{\jbar}{\bar{j}}
\newcommand{\kbar}{\bar{k}}
\newcommand{\lbar}{\bar{l}}
\newcommand{\zbar}{\bar{z}}

\newcommand{\fbar}{\bar{f}}
\newcommand{\wbar}{\bar{w}}

\newcommand{\abar}{\bar{\alpha}}
\newcommand{\bbar}{\bar{\beta}}

\newcommand{\sbar}{\bar{\sigma}}
\newcommand{\gamabar}{\bar{\gamma}}

\DeclareMathOperator{\trace}{trace}

\begin{document}
\title[The Webster scalar curvature and sharp upper and lower bounds]{The Webster scalar curvature and sharp upper and lower bounds for the first positive eigenvalue of the Kohn-Laplacian on real hypersurfaces}

\author{Song-Ying Li}
\address{Department of Mathematics, University of California, Irvine, CA 92697-3875 and Department of Math, Fujian Normal University, Fuzhou, Fujian 350108}
\email{sli@math.uci.edu}
\author{Duong Ngoc Son}
\address{Fakult\"{a}t f\"{u}r Mathematik, Universit\"{a}t Wien, Oskar-Morgenstern-Platz 1, 1090 Wien, Austria}
\email{son.duong@univie.ac.at}
\thanks{The second-named author was supported by the Austrian Science Fund FWF, Project I01776.}

\date{March 20, 2018}
\begin{abstract}
	Let $(M,\theta)$ be a compact strictly pseudoconvex pseudohermitian manifold which is CR embedded into a complex space.
	In an earlier paper, Lin and the authors gave several sharp upper bounds for the first positive 
	eigenvalue $\lambda_1$ of the Kohn-Laplacian $\Box_b$ on $(M,\theta)$. In the present paper, we give a sharp upper bound for $\lambda_1$, generalizing and extending some previous results. As a corollary, we obtain a Reilly-type estimate when $M$ is embedded into the standard sphere. In another direction, using a Lichnerowicz-type estimate by Chanillo, Chiu, and Yang and an explicit formula for the Webster scalar curvature, we give a lower bound for $\lambda_1$ when the pseudohermitian structure $\theta$ is volume-normalized.
\end{abstract}
\maketitle

\section{Introduction}
In recent years, there has been much effort devoted to the study of the spectrum of the Kohn-Laplacian $\Box_b$ on \emph{compact} strictly pseudoconvex CR manifolds; see e.g. \cite{CCY,LSW,DLL}. It is proved by Burns-Epstein \cite{BE} for the three-dimensional case that the spectrum of $\Box_b$ in $(0, \infty)$ consists of point eigenvalues of finite multiplicity and the corresponding eigenfunctions are smooth (the higher dimensional case is well-known and even simpler; see e.g. \cite{LSW}). Moreover, zero is an isolated eigenvalue if and only if $\Box_b$ has closed range. By the work of Kohn \cite{Kohn}, the latter condition is satisfied if $M$ is embeddable in a complex space. Furthermore, this embedability condition holds if the manifold is compact and of dimension at least five \cite{BdM}. Thus, assume that $M$ is embedded, the spectrum of $\Box_b$ consists of zero and a sequence of point eigenvalues $0< \lambda_1 < \lambda_2 < \cdots < \cdots $ increasing to the infinity. In this situation, one may study the basic problem of estimating the first positive eigenvalue $\lambda_1$ on a compact embeddable strictly pseudoconvex pseudohermitian manifold. In a recent paper \cite{DLL}, Guijuan Lin and the authors gave several sharp and explicit upper bounds for $\lambda_1$ on compact strictly pseudoconvex real hypersurfaces of $\mathbb{C}^{n+1}$. It is proved in \cite[Theorem 1.2]{DLL} that when $M$ is defined by a \emph{strictly} plurisubharmonic defining function $\rho$ which satisfies a certain condition, and $\theta : = (i/2)(\bar{\partial}\rho - \partial \rho)$, then $\lambda_1$ is bounded above by the average value on $M$ of $|\partial\rho|^{-2}_{\rho}$, the inverse of the square of the length of $\partial \rho$ in the K\"ahler metric induced by $\rho$ near $M$. In Theorem 4.1 of the same paper, another upper bound is given in terms of $|\partial\rho|^{-2}_{\rho}$ and the eigenvalues of the complex Hessian $H[\rho]$ of $\rho$. The first purpose of the present paper is to continue studying the upper bounds for $\lambda_1$, explicitly in terms of $\rho$. Namely, we shall prove that a similar bound still holds when $\rho$ satisfies a more general condition given below. We also relate the bound with the \emph{transverse curvature} of $M$ induced by $\rho$. To describe our results in detail, let us first fix some notations. Let $M\subset \mathbb{C}^{n+1}$ be a strictly pseudoconvex real hypersurface and $\rho$ a defining function for $M$, i.e., \(M:=\{Z\in \mathbb{C}^{n+1} \colon \rho(Z)=0\}\), with $d\rho \ne 0$ along $M$. The restriction of $ i \partial\bar{\partial} \rho$ to $T^{1,0}M\times T^{0,1}M$ is definite and hence assumed to be positive. There exists a unique $(1,0)$ vector field $\xi$ near $M$ and a real function $r$ such that \cite{LM82}
\begin{equation}
\xi \lrcorner\, \partial\bar{\partial} \rho = r[\rho] \, \bar{\partial} \rho,
\quad
\partial\rho(\xi) = 1.
\end{equation}
The function $r=r[\rho]$ is called \emph{transverse curvature} by Graham and Lee \cite{GL, LM82}. The readers can check that if $\rho$ is strictly plurisubharmonic, then $r[\rho] = |\partial\rho|^{-2}_{\rho}$. This function will play an important role in this paper.

On the other hand, the defining function $\rho$ ``induces'' a pseudohermitian structure $\theta$ on $M$ by
\begin{equation}
\theta = \frac{i}{2}(\bar{\partial}\rho - \partial \rho).
\end{equation}
Thus $(M,\theta)$ is a pseudohermitian manifold in the sense of Webster \cite{Webster}. Since $M$ is strictly pseudoconvex, $dV: = \theta \wedge (d\theta)^n$ is a volume form on $M$. This is the volume form we use to define the adjoint $\bar{\partial}^{\ast}_b$ of the $\bar{\partial}_b$ operator. The Kohn-Laplacian acting on function is thus $\Box_b:= \bar{\partial}_b ^{\ast} \bar{\partial}_b$.

The first result in this paper is the following theorem.
\begin{theorem}\label{thm:ub}
	Let $M$ be a compact strictly pseudoconvex hypersurface in $\mathbb{C}^{n+1}$ defined by $\rho = 0$ and let $\theta = \iota^{\ast} (i/2)(\bar{\partial} \rho - \partial\rho)$. 
	Suppose that there are positive numbers $N>0$, $\nu>0$, and a pluriharmonic function $\psi$ defined in a neighborhood of $M$ such that
	\begin{equation}\label{e:13}
	(\rho + \nu)^N - \psi = \sum_{\mu = 1}^K |f^{(\mu)}|^2.
	\end{equation}
	where $f^{\mu}$ are holomorphic for $\mu = 1,2,\dots , K$. Then
	\begin{equation}\label{e:1.4}
	\lambda_1(M,\theta) \leq \frac{n}{v(M,\theta)} \int_M r[\rho]\, \theta \wedge (d\theta)^n + \frac{n(N-1)}{\nu}.
	\end{equation}
	If the equality occurs, then the functions $b^{\mu}:=\Box_b \fbar^{(\mu)}$ are eigenfunctions that correspond to~$\lambda_1$.
\end{theorem}
It is easy to see that in the special case $N=1$, estimate~\eqref{e:1.4} is stronger than the one in Theorem~4.1 in \cite{DLL}.
Moreover, putting $N=1$, $\nu=1$, and $\psi = 0$, we obtain the following ``Reilly-type'' estimate.
\begin{corollary}\label{cor:1}
	Suppose that $M$ is a compact strictly pseudoconvex manifold and $F\colon M\to \mathbb{S}^{2K+1}$ is a CR immersion. Let $\Theta$ be the standard pseudohermtian structure on the unit sphere, and let $r_F =r[\rho]$ where \(\rho: = \sum_{\mu = 1}^K |F^{(\mu)}|^2 -1\). Then
	\begin{equation}\label{e:ub}
	\lambda_1(M,F^{\ast} \Theta) \leq \frac{n}{v(M,F^{\ast} \Theta)} \int_M r_F\, F^{\ast} \Theta \wedge (dF^{\ast} \Theta)^n.
	\end{equation}
	If the equality occurs, then $b^{(\mu)}:=\Box_b\bar{F}^{(\mu)}$ are eigenfunctions that correspond to~$\lambda_1$.
\end{corollary}
We shall show in Example~\ref{ex:1} that $r_F$ is constant does not implies that the equality in \eqref{e:ub} holds. It is an open question whether the equality in \eqref{e:ub} holds only if $M$ is a sphere.

We note that Corollary 1.2 is analogous to the well-known Reilly extrinsic estimate in Riemannian geometry from 1977 \cite{Reilly}: ``The first eigenvalue of the Laplacian for a compact $n$-manifold isometrically immersed in Euclidean space is bounded above by $n$ times the average value of the square of the norm of the mean curvature vector.'' 

The second purpose of this paper is to study the lower bound for $\lambda_1$, explicitly in terms of the defining function. For this purpose, we shall
consider the unique volume-normalized pseudohermitian structure for $M$. Recall that for any strictly pseudoconvex hypersurface $M \subset \mathbb{C}^{n+1}$, there exists a unique pseudohermitian structure $\Theta$ on $M$ that is \emph{volume-normalized} with respect to the restriction of $\zeta:= dz^1 \wedge dz^2 \wedge \cdots \wedge dz^{n+1}$ onto $M$. It is well-known (cf. \cite{Farris}) that if $\rho$ is a defining function satisfying $J[\rho] =1$ on $M$, then $\Theta: = (i/2)(\bar{\partial}\rho - \partial\rho)$ is volume-normalized with respect to $\zeta$. Here, $J[\rho]$ be the (Levi-) Fefferman determinant of $\rho$:
\begin{equation}
J[\rho] = - \det \begin{bmatrix}
\rho & \rho_{\kbar} \\
\rho_{j} & \rho_{j\kbar}
\end{bmatrix}.
\end{equation}
Following Li \cite{L16}, we say that $M$ is \emph{super-pseudoconvex} if the Webster scalar curvature $R_{\Theta}$ of $\Theta$ is positive. This condition is equivalent to the fact that the approximate solution to the Fefferman equation is strictly plurisubharmonic near $M$ (see \cite{L16}). For any defining function $\rho$ of $M$, we define
\begin{equation}\label{e:l2}
D[\rho]: = n(n+1)r[\rho] - n N_{\rho} \log J[\rho] - \tfrac{1}{2}\Delta_b \log J[\rho] - \tfrac{n}{n+1}|\bar{\partial}_b \log J[\rho]|^2 .
\end{equation}
where $N_{\rho}: = \tfrac{\xi + \bar{\xi}}{2}$ is a real vector field transverse to $M$. The sub-Laplacian $\Delta_b$, the tangential Cauchy-Riemann operator $\bar{\partial}_b$, and the norm $|\cdot |$, are defined using $\theta:=(i/2)(\bar{\partial}\rho - \partial\rho)$. The dependency of $D[\rho]$ on $\rho$ is rather complicated. However, it is clear from \eqref{e:l2} that if $J[\rho] = 1 + O(\rho^2)$ near $M$ then $D[\rho] = n(n+1)r[\rho]$ on $M$.

Our next result is the following theorem.
\begin{theorem}\label{thm:lb}
	Suppose that $M$ is a strictly pseudoconvex real hypersurface in $\mathbb{C}^{n+1}$ and $\Theta$ is the volume-normalized pseudohermitian structure on $M$. Then
	\begin{enumerate}[(i)]
		\item For any defining function $\rho$ of $M$ with $J[\rho]>0$, the Webster scalar curvature of $(M,\Theta)$ is given by
		\begin{equation}\label{e:scal}
		R_{\Theta} = J[\rho]^{1/(n+1)} D[\rho].
		\end{equation}
		In particular, the expression on the right hand side does not depend on $\rho$.
		\item $M$ is super-pseudoconvex if and only if $D[\rho] > 0$ on $M$ for a defining function $\rho$ with $J[\rho] > 0$ near $M$ (and hence for all such $\rho$).
	\end{enumerate}
\end{theorem}

Note in passing that in case $n =1$, Hammond \cite{Hammond} also gave a differential operator $\mathcal{L}$ explicitly such that for 
any defining function $\rho$, $\mathcal{L}[\rho]$ is a constant multiple of the Webster scalar curvature $R_\Theta$; see Remark~4.2.

We obtain the following corollary which gives an explicit lower bound for the first positive eigenvalue of the Kohn-Laplacian.
\begin{corollary}
	Let $M$ be a compact strictly pseudoconvex real hypersurface of $\mathbb{C}^{n+1}$ and $\Theta$ the
	Fefferman volume-normalized pseudohermitian structure on $M$. Let $\lambda_1$ be the first positive eigenvalue of $\Box_b$ on $(M,\Theta)$. Then for any defining function $\rho$ with 
	$J[\rho]>0$,
	\begin{equation}\label{e:lb}
	\lambda_1 \geq n\min_M J[\rho]^{1/(n+1)} D[\rho].
	\end{equation}
	provided that $n\geq 2$, or $n=1$ and $M$ has positive CR Paneitz operator.	In particular, if $\rho$ is a Fefferman approximate solution, then
	\begin{equation}
	\lambda_1 \geq n\min_M \det H[\rho]=n \min_M r[\rho].
	\end{equation}
	If the equality holds and $n\geq 2$, then $M$ must be equivalent to a sphere.
\end{corollary}	
We refer the reader to \cite{CCY} for the definition of CR Paneitz operator. The question when a compact strictly pseudoconvex real hypersurface in $\mathbb{C}^2$ has positive CR Paneitz operator is still open in general.

We note that, for the special case when $M$ is the real ellipsoid, the sharp lower bound for $\lambda_1$ and the characterization of equality case (Obata type theorem) were given by Lin \cite{Lin}.

The paper is organized as follows. In Section 2, we collect some basic facts about pseudohermitian geometry and fix notations. In Section 3, we prove a general formula for $\Box_b$ generalizing a formula in \cite{DLL} and prove Theorem 1.1. In Section 4, we prove Theorem 1.3 and Corollary~1.4.

\section{Preliminaries}

Let $M\subset \mathbb{C}^{n+1}$ be a strictly pseudoconvex real hypersurface and $\rho$ a defining function for $M$, i.e.,
\(M:=\{Z\in \mathbb{C}^{n+1} \colon \rho(Z)=0\},\) with $d\rho \ne 0$ along $M$. The restriction of $ i\partial\bar{\partial} \rho$ to $T^{1,0}M\times T^{0,1}M$ is definite and is assumed to be positive. There exists unique $(1,0)$ vector field $\xi$ near $M$ and a real function $r$ such that, cf. \cite{LM82}
\begin{equation}
\xi\, \lrcorner \, \partial\bar{\partial} \rho = r[\rho] \, \bar{\partial} \rho,
\quad
\partial\rho(\xi) = 1.
\end{equation}
Let \(N = \tfrac12\left(\xi + \overline{\xi}\right)\) and \(T = i(\xi - \overline{\xi})\). Then \(N \rho = 1\) and \(T \rho = 0\). Furthermore, $T$ is the Reeb vector field associated to the pseudohermitian induced by $\rho$: 
\begin{equation}
\theta = \frac{i}{2}(\bar{\partial}\rho - \partial \rho).
\end{equation}
In local computations, write $\xi = (\xi^k)$ and choose
\begin{equation}\label{e:cframe}
\theta^{k}: = dz^k - \xi^k \partial\rho.
\end{equation}
Since $\rho_k \xi^k = 1$, one has \(\rho_k \theta^k = 0\). By direct calculations,
\begin{align}
-id\theta
= 
\rho_{j\kbar} \theta^{j} \wedge \theta^{\kbar}
+ 
r [\rho] \partial\rho \wedge \bar{\partial}\rho.
\end{align}
Thus, the restriction of $\rho_{j\kbar}\theta^j\wedge \theta^{\kbar}$ to $M$ is the Levi-form of $(M,\theta)$. The function $r=r[\rho]$
is called \emph{transverse curvature} \cite{GL}. When $\rho_{n+1} \ne 0$, we can write
\begin{equation}
-id\theta
= h_{\alpha\bbar} \theta^{\alpha}\wedge \theta^{\bbar}
+
r \partial\rho \wedge \bar{\partial}\rho,
\end{equation}
where
\begin{equation}\label{e:levimatrix}
h_{\alpha\bbar}=\rho_{\alpha \bbar}-\rho_\alpha \partial_{\bbar}\log \rho_{n+1}-\rho_{\bbar}\partial_{\alpha}\log \rho_{\overline{n+1}}+\rho_{n+1\overline{n+1}}
\frac{\rho_\alpha \rho_{\bbar}}{|\rho_{n+1}|^2}
\end{equation}
are the entries of the Levi matrix. By direct calculations, we obtain 
\begin{equation}
h_{\alpha\bbar} \xi^{\alpha} = -Z_{\bbar} \log \rho_{\overline{n+1}}\quad \hbox{with }\ Z_{\bbar}={\partial \over \partial \zbar_\beta}-{\rho_{\bbar}\over \rho_{\overline{n+1}}}{\partial \over
	\partial \zbar_{n+1}}.
\end{equation}
Sometimes we use $w$ for the last coordinate $z_{n+1}$. We always assume that $J[\rho] > 0$ along $M$.

\begin{proposition} Let $\phi^{j\kbar}$ be the adjugate matrix of $H[\rho]: = [\rho_{k\lbar}]$.	Then
	\begin{align}
	r[\rho]\, \rho_{\kbar} = \rho_{j\kbar} \xi^{j}, \quad r[\rho] = \rho_{j\kbar}\xi^{j}\xi^{\kbar} = \frac{\det H[\rho]}{J[\rho]}, \quad 
	\hbox{and }\ \xi^k = \frac{\phi^{\jbar k} \rho_{\jbar}}{J[\rho]}.
	\end{align}
	In particular, $\det H[\rho] = 0$ if and only if the transverse curvature $r[\rho]=0$. 
\end{proposition}
\begin{proof}
	This is from routine calculation.
\end{proof}
To compute the connection forms, we note that $\theta^{\alpha} = dz^{\alpha} - i\xi^{\alpha}\theta$ is an admissible coframe on $M$ with respect to $\theta$. 
Then the corresponding connection forms $\omega_{\beta}{}^{\alpha}$ were calculated in \cite{LL} (see also \cite{Webster}). Namely,
\begin{align}\label{e:cf}
\omega_{\beta}{}^{\alpha}
=
\left(h^{\alpha\bar{\mu}}Z_{\gamma} h_{\beta\bar{\mu}} -\xi_{\beta}\delta_{\gamma}^{\alpha}\right) \theta^{\gamma}
+ \xi^{\alpha} h_{\beta\bar{\gamma}} \theta^{\bar{\gamma}}
- i Z_{\beta} \xi^{\alpha}\theta.
\end{align}
This formula will be crucial for our further computations. 

\begin{proposition}[Li-Luk \cite{LL}]\label{prop:ll}
	Let $M$ be a strictly pseudoconvex real hypersurface defined by $\rho=0$, not necessarily plurisubharmonic, and $\theta = (i/2)(\bar{\partial}\rho - \partial\rho)$. Define
	\begin{equation}
	D^{\rho}_{\alpha\bbar}
	=
	\partial_{\bbar}\partial_{\alpha} 
	-
	(\rho_{\alpha}/\rho_{w}) \partial_{w}\partial_{\bbar}
	-
	(\rho_{\bbar}/\rho_{\wbar}) \partial_{\wbar}\partial_{\alpha} 
	+
	(\rho_{\alpha}\rho_{\bbar}/|\rho_{w}|^2) \partial_{w}\partial_{\wbar}.
	\end{equation}
	Then in term of the coframe $\{\theta^{\alpha}\}$ as in \eqref{e:cframe}, the Webster Ricci tensor has components
	\begin{equation}
	R_{\alpha\bbar}
	=
	- D^{\rho}_{\alpha\bbar} \log J[\rho]
	+ (n+1) r[\rho]\, h_{\alpha\bbar}.		
	\end{equation}
\end{proposition}

\begin{proof} This follows from \cite{LL}, equation (2.43) and the fact that $r[\rho] = \det H(\rho) / J[\rho]$.
\end{proof}
The following proposition is implicit in \cite{LM82}.
\begin{proposition}
	Let \(\psi_{j\kbar} =\rho_{j\kbar} + (1-r[\rho])\rho_j \rho_{\kbar}\) and suppose that $J[\rho] >0$. Then $\psi_{j\kbar}$ is invertible. Let $\psi^{\kbar j}$ be its inverse. Then \(\xi^{\kbar} = \rho_j \psi^{\kbar j}\).	Moreover, the inverse of Levi-matrix is
	\begin{equation}\label{e:inverselevi}
	h^{\bbar\gamma} 
	=
	\psi^{\bbar\gamma} - \xi^{\bbar} \xi^{\gamma}.
	\end{equation}
\end{proposition}

\begin{proof}
	That $\psi_{j\kbar}$ is invertible is already observed in \cite{LM82}. In fact, by a simple formula for determinant of rank-one perturbations of a matrix,
	\begin{align}
	\det [\psi_{j\kbar}] 
	& = 
	\det H[\rho] + (1-r)\rho_{\kbar}\,\phi^{\kbar j} \rho_j \notag \\
	& = 
	\det H[\rho] + (1-r) J[\rho] = J[\rho].
	\end{align}
	This implies that $\psi_{j\kbar}$ is invertible since $J[\rho] > 0$ on $M$. 
	
	Since $r[\rho] \rho_j = \rho_{j\kbar} \xi^{\kbar}$ and $\rho_{\kbar}\xi^{\kbar} = 1$, we obtain
	\begin{align}
	r\rho_j \psi^{j\lbar}
	= 
	\rho_{j\kbar} \xi^{\kbar}\psi^{j\lbar} 
	= 
	\xi^{\kbar}(\delta_{\kbar}^{\lbar} - (1-r)\rho_{j}\rho_{\kbar}\psi^{j\lbar}) 
	= 
	\xi^{\lbar} - (1-r)\rho_j\psi^{j\lbar}
	\end{align}
	Therefore,
	\begin{equation}
	\psi^{j\lbar} \rho_{j} = \xi^{\lbar}.
	\end{equation}
	To prove \eqref{e:inverselevi}, with the notation $w=n+1$, we note that
	\begin{equation}
	\xi^{\gamma}h_{\gamma\abar} = -Z_{\abar}\log \rho_{w}.
	\end{equation}
	On the other hand,
	\begin{align}
	\psi^{\bbar\gamma} \rho_{\gamma\abar}
	& = 
	\psi^{\bbar j} \rho_{j\abar} - \psi^{\bbar w}\rho_{w\abar} \notag \\
	& = 
	\delta_{\abar}^{\bbar} - (1-r)\rho_{\abar}\rho_{j}\psi^{\bbar j} - \psi^{\bbar w}\rho_{w\abar} \notag \\
	& = 
	\delta_{\abar}^{\bbar} - (1-r)\rho_{\abar}\xi^{\bbar} - \psi^{\bbar w}\rho_{w\abar}.
	\end{align}
	Here as usual, repeated Latin indices are summed from $1$ to $n+1$ and Greek indices are summed from $1$ to $n$. Also by direct calculation,
	\begin{align}
	\frac{\psi^{\bbar\gamma} \rho_{\gamma}\rho_{w\abar}}{\rho_{w}}
	& = 
	\frac{\rho_{w\abar} \xi^{\bbar} \rho_{\wbar}}{|\rho_w|^2} - \psi^{\bbar w}\rho_{w\abar},\\
	\frac{\psi^{\bbar\gamma} \rho_{\abar} \rho_{\gamma\wbar}}{\rho_{\wbar}}
	& = 
	-\frac{\psi^{\bbar w}\rho_{w\wbar}\rho_{\abar}}{\rho_{\wbar}}- (1-r) \rho_{\abar} \xi^{\bbar},\\
	\frac{\psi^{\bbar\gamma}\rho_{\abar}\rho_{\gamma}\rho_{w\wbar}}{|\rho_w|^2} 
	& = 
	\frac{\rho_{\abar} \xi^{\bbar}\rho_{w\wbar}}{|\rho_w|^2} - \frac{\rho_{\abar}\rho_{w\wbar} \psi^{\bbar w}}{\rho_{\wbar} }.
	\end{align}
	Thus,
	\begin{align}
	\psi^{\bbar\gamma}h_{\gamma\abar}
	= 
	\delta_{\abar}^{\bbar}
	+
	\frac{\rho_{\abar}\xi^{\bbar}\rho_{w\wbar} - \rho_{w\abar} \xi^{\bbar} \rho_{\wbar}}{|\rho_w|^2} 
	= 
	\delta_{\abar}^{\bbar} 
	-
	\xi^{\bbar} Z_{\abar} \log \rho_{w}
	\end{align}
	Therefore,
	\begin{equation*}
	(\psi^{\bbar\gamma} - \xi^{\bbar}\xi^{\gamma})\, h_{\gamma\abar} = \delta_{\abar}^{\bbar}.\qedhere
	\end{equation*}
\end{proof}
\section{An explicit formula for the Kohn-Laplacian and the proof of Theorem~\ref{thm:ub}}
Explicit formulas for the Kohn-Laplacian have been given in \cite{DLL}; see also \cite{LM82}. Here we prove the formula given in \cite{DLL} in a more general situation. We shall use the notations as in Section~2.

\begin{proposition} Let $ \widetilde{\Delta}_{\rho}$ be the degenerate second order real differential operator
	\begin{equation}
	\widetilde{\Delta}_{\rho} = (\xi^j \xi^{\kbar} - \psi^{\kbar j}) \partial_{j}\partial_{\kbar}.
	\end{equation}
	The Kohn-Laplacian and the sub-Laplacian on $(M,\theta)$ are given by
	\begin{align}\label{e:ks}
	\Box_b f
	=
	\widetilde{\Delta}_{\rho} f + n\, \bar{\xi} f,
	\quad \tfrac{1}{2}\Delta_b u
	=
	\widetilde{\Delta}_{\rho} u + n\, N_{\rho} u,
	\end{align}
	where $f$ is complex-valued and $u$ is real-valued.
\end{proposition}
\begin{proof}
	The proof is similar to that of Proposition 2.1 in \cite{DLL} which, in turns, 
	uses the formula for the Tanaka-Webster connection forms obtained by Li-Luk~\cite{LL}. Here we work in a more general case with $\rho$ is not assumed to be strictly plurisubharmonic.
	
	First, notice that
	\begin{align}
	Z^{\bbar} = h^{\bbar\gamma} Z_{\gamma}
	& = 
	(\psi^{\bbar\gamma} - \xi^{\bbar}\xi^{\gamma})\left(\partial_{\gamma} - \frac{\rho_{\gamma}}{\rho_{w}}\partial_w\right) \notag\\
	& = 
	\psi^{\bbar\gamma} \partial_{\gamma} 
	-
	\frac{\psi^{\bbar\gamma}\rho_{\gamma}}{\rho_w}\partial_w 
	-
	\xi^{\bbar} \xi^{\gamma} \partial_{\gamma}
	+
	\frac{\xi^{\bbar}\xi^{\gamma}\rho_{\gamma}}{\rho_w}\partial_w \notag\\
	& = 
	\psi^{\bbar\gamma} \partial_{\gamma} 
	-
	\left(\frac{\xi^{\bbar} - \psi^{\bbar w}\rho_{w}}{\rho_w}\right)
	\partial_{w}
	-
	\xi^{\bbar} \xi^{\gamma} \partial_{\gamma}
	+
	\left(\frac{\xi^{\bbar}(1-\xi^{w}\rho_w)}{\rho_w}\right)\partial_w \notag\\
	& = 
	\psi^{\bbar\gamma} \partial_{\gamma} + \psi^{\bbar w}\partial_w - \xi^{\bbar} \xi^{\gamma}\partial_{\gamma} - \xi^{\bbar} \xi^{w}\partial_{w} \notag\\
	& = 
	\psi^{\bbar k}\partial_{k} - \xi^{\bbar} \xi^k \partial_k.
	\end{align}
	By \eqref{e:cf}, (2.18), (2.7), and $\omega_{\bbar}{}^{\sbar}(Z_{\gamma}) = \xi^{\sbar} h_{\bbar\gamma}$, we obtain
	\begin{align}
	-\Box_b f
	& = 
	h^{\bbar\gamma} \left(Z_{\gamma}Z_{\bbar} f - \omega_{\bbar}{}^{\sbar}(Z_{\gamma}) Z_{\sbar} f\right)\notag\\
	& = 
	Z^{\bbar}Z_{\bbar} f - h^{\bbar\gamma}\xi^{\sbar} h_{\bbar\gamma} Z_{\sbar} f \notag \\
	& = 
	\left[\psi^{\bbar k}\partial_{k} - \xi^{\bbar} \xi^k \partial_k\right] \left[f_{\bbar} - \frac{\rho_{\bbar}}{\rho_{\wbar}} f_{\wbar}\right]
	-n \xi^{\sbar} Z_{\sbar} f \notag \\
	& = -\tilde{\Delta} f-
	\left[\psi^{\wbar k}f_{k\wbar} - \xi^{\wbar} \xi^k f_{k\wbar}\right] \notag\
	-\left[\psi^{\bbar k}f_{k\wbar} - \xi^{\bbar} \xi^k f_{k\wbar} \right] \frac{\rho_{\bbar}}{\rho_{\wbar}} \\
	& \quad -{1\over \rho_{\wbar}}
	\left[\psi^{\bbar k}\rho_{k\bbar} - \xi^{\bbar} \xi^k \rho_{k\bbar}\right] f_{\wbar} 
	+{1\over \rho_{\wbar}^2} \left[\psi^{\bbar k}\rho_{k \wbar} - \xi^{\bbar} \xi^k \rho_{k\wbar}\right] \rho_{\bbar} f_{\wbar}-n \xi^{\sbar} Z_{\sbar} f\notag\\
	& = -\tilde{\Delta} f+ \xi^k f_{k\wbar} \left[ \xi^{\wbar} 
	- \frac{1}{\rho_{\wbar}} + \xi^{\bbar} \frac{\rho_{\bbar}}{\rho_{\wbar}}\right] \notag\\
	& \quad -{f_{\wbar} \over \rho_{\wbar}}\Big(
	\left[\psi^{\bbar k} - \xi^{\bbar} \xi^k \right] \rho_{k\bbar} 
	-{1\over \rho_{\wbar}} \left[\psi^{\bbar k} - \xi^{\bbar} \xi^k \right] \rho_{k\wbar} \rho_{\bbar} \Big) -n \xi^{\sbar} Z_{\sbar} f \notag\\
	& = -\tilde{\Delta} f -n \xi^{\sbar} Z_{\sbar} f \notag\\
	& \quad - {f_{\wbar} \over \rho_{\wbar}}\Big( n-(1-r) \xi^{\bbar} \rho_{\bbar} - \xi^{\bbar} r \rho_{\bbar} 
	+{\rho_{\bbar} \over \rho_{\wbar}} \left[(1-r) \xi^{\bbar} \rho_{\wbar} + \xi^{\bbar} r \rho_{ \wbar}\right] \Big) \notag \\
	& = -\tilde{\Delta} f -n \xi^{\sbar} f_{\sbar} +n \xi^{\sbar} {\rho_{\sbar} \over \rho_{\wbar}} f_{\wbar}
	- {f_{\wbar} \over \rho_{\wbar}}\Big( n - \xi^{\bbar} r \rho_{\bbar} 
	+\rho_{\bbar} \xi^{\bbar} r \Big) \notag\\
	& = -\tilde{\Delta} f -n \xi^{\sbar} f_{\sbar} -n \xi^{\wbar} f_{\wbar}.
	\end{align}
	This proves the formula for $\Box_b$ in \eqref{e:ks}. The one for $\Delta_b$ can be obtained by noting that $\Delta_b u = \Box_b u + \overline{\Box_b u}$ ($u$ is real-valued). 
\end{proof}
\begin{proof}[Proof of Theorem~\ref{thm:ub}] Differentiate \eqref{e:13}, we obtain
	\begin{equation}
	N\nu^{N-1} \rho_j - \psi_j
	=
	\sum_{\mu = 1}^{K} f_j^{(\mu)} \fbar^{(\mu)},
	\quad
	N\nu^{N-2}(\nu\rho_{j\kbar} + (N-1)\rho_j\rho_{\kbar})
	=
	\sum_{\mu = 1}^{K} f_j^{(\mu)} \fbar^{(\mu)}_{\kbar}.
	\end{equation} 
	On the other hand, by the formula for the Kohn-Laplacian in \eqref{e:ks}, 
	\(\Box_b \bar{f}^{(\mu)} =n\xi^{\kbar} \bar{f}^{(\mu)}_{\kbar}\). Therefore,
	\begin{equation}
	|\Box_b \bar{f}^{(\mu)}|^2 = n^2 \xi^{\kbar} \bar{f}^{(\mu)}_{\kbar} \xi^{l} f^{(\mu)}_{l}.
	\end{equation}
	Summing over $\mu=1,2,\dots , K$, we obtain
	\begin{align}
	\sum_{\mu=1}^K |\Box_b \bar{f}^{(\mu)}|^2
	& =
	n^2 \xi^{\kbar}\xi^l \sum_{\mu=1}^K \bar{f}^{(\mu)}_{\kbar} f^{(\mu)}_{l} \notag \\
	& =
	n^2 \xi^{\kbar}\xi^{l}\left(N\nu^{N-1} \rho_{l\kbar} + N(N-1)\nu^{N-2}\rho_j \rho_{\kbar}\right) \notag \\
	& =
	n^2 N\nu^{N-2}\,\left(\nu\, r + N-1\right).
	\end{align}	
	Next, observe that by \eqref{e:levimatrix} (assuming $\rho_w \ne 0$)
	\begin{align}
	\sum_{\mu = 1}^{K} Z_{\bar{\gamma}}\bar{f}^{(\mu)} Z_{\sigma} {f}^{(\mu)}%
	& = 
	\sum_{\mu = 1}^{K} \left(\bar{f}^{(\mu)}_{\bar{\gamma}} - \frac{\rho_{\bar{\gamma}}}{\rho_{\bar{w}}} \bar{f}^{(\mu)}_{\wbar}\right) 
	\left(f^{(\mu)}_{\sigma} - \frac{\rho_{\sigma}}{\rho_{w}} f^{(\mu)}_{w}\right) \notag \\
	& = 
	\sum_{\mu = 1}^{K} 
	\bar{f}^{(\mu)}_{\bar{\gamma}} f^{(\mu)}_{\sigma}
	-
	\frac{\rho_{\bar{\gamma}}}{\rho_{\bar{w}}} 
	\sum_{\mu = 1}^{K} \bar{f}^{(\mu)}_{\wbar} f^{(\mu)}_{\sigma}
	-
	\frac{\rho_{\sigma}}{\rho_{w}} \sum_{\mu = 1}^{K} f^{(\mu)}_{w}
	\bar{f}^{(\mu)}_{\bar{\gamma}}
	+
	\frac{\rho_{\bar{\gamma}}\rho_{\sigma}}{|\rho_w|^2} 
	\sum_{\mu = 1}^{K}f^{(\mu)}_{w} \bar{f}^{(\mu)}_{\wbar} \notag \\
	& = 
	N\nu^{N-1}\left(\rho_{\sigma\gamabar} - \frac{\rho_{\gamabar}\rho_{\sigma \wbar}}{\rho_{\wbar}} - \frac{\rho_{\sigma}\rho_{\gamabar w}}{\rho_{w}} + \frac{\rho_{\gamabar}\rho_{\sigma}\rho_{w\wbar}}{|\rho_w|^2}\right) \notag \\
	& =
	N \nu^{N-1} h_{\sigma\bar{\gamma}}.
	\end{align}
	Therefore,
	\begin{equation}
	\sum_{\mu = 1}^{K} |\bar{\partial}_b \bar{f}^{(\mu)}|^2
	= 
	h^{\sigma\bar{\gamma}}\sum_{\mu = 1}^{K} Z_{\bar{\gamma}}\bar{f}^{(\mu)} Z_{\sigma} {f}^{(\mu)}
	= n N\nu^{N-1}.
	\end{equation}
	Applying Corollary~3.2 in \cite{DLL}, we obtain,
	\begin{equation}
	\lambda_1
	\leq 
	\min_{\mu}
	\frac{\|\Box_b \bar{f}^{(\mu)}\|^2}{\|\bar{\partial}_b \bar{f}^{(\mu)}\|^2}
	\leq 
	\frac{\sum_{\mu=1}^K \|\Box_b \bar{f}^{(\mu)}\|^2}{\sum_{\mu=1}^K\|\bar{\partial}_b \bar{f}^{(\mu)}\|^2}
	=
	\frac{n}{v(M)}\int_M r\, \theta \wedge (d\theta)^n + \frac{n(N-1)}{\nu}.
	\end{equation}
	This proves the inequality.
	
	Next, assume that the equality occurs in \eqref{e:1.4} and $b^{(\mu)}: = \Box_b \fbar^{(\mu)}$. Then by inspecting the proof
	of Corollary~3.2 in \cite{DLL}, $b^{(\mu)}$ must be orthogonal to the eigenfunctions corresponding to
	$\lambda_k$ for all $k\geq 2$. Also, it is clear that $b^{(\mu)}$ is orthogonal to $\ker \Box_b$. Consequently, $b^{(\mu)}$ must be in the eigenspace that corresponds to $\lambda_1$, i.e., $\Box_b b^{(\mu)} = \lambda_1 b^{(\mu)}$.
	
	Finally, $b^{(\mu)}$ is non-trivial, as otherwise, $\fbar^{(\mu)}$ must be CR and hence $f^{(\mu)}$ must be a constant. This is a contradiction.
\end{proof}
We remark that \eqref{e:ks} generalizes a formula in \cite{DLL}: We do not assume here that $\rho$ is strictly plurisubharmonic. Using this formula, we can slightly improves Theorem 1.1 in \cite{DLL} as follows.
\begin{theorem}
	Suppose $M$ is a compact strictly pseudoconvex hypersurface given by $\rho=0$ with the transverse curvature $r[\rho] \geq 0$. Assume that for some $j$,
	\begin{align}\label{special}
	\Re \left( n\,r[\rho]\, \rho_{\jbar}\tilde{\Delta}_{\rho}\, \rho_{j} + |\widetilde{\Delta}_{\rho}\, \rho_j|^2\right) \leq 0 \quad \text{on} \ M.
	\end{align}
	Then
	\begin{align}\label{specialbound}
	\lambda_1(M,\theta) \leq n\max_M r[\rho].
	\end{align}
	and the equality holds only if the transverse curvature $r[\rho]$ is constant along $M$.
\end{theorem}
We end this section by the following example showing that the constancy of the transverse curvature 
does not implies the equality in the estimates, even in the case $M$ is a sphere.
\begin{example}\label{ex:1}\rm 
	The unit sphere $\mathbb{S}^3$ in $\mathbb{C}^2$ can be defined by $\rho = 0$ with
	\begin{equation}
	\rho = |z^2|^2 + 2|zw|^2 + |w^2|^2 -1.
	\end{equation}
	Observe that on $\mathbb{S}^3$, $\det H[\rho] = 8$, $J[\rho] = 4$, and the transverse curvature is constant: $r[\rho] = 2$.	Since $\partial\rho = 2(\zbar dz + \wbar dw)$ on $\mathbb{S}^3$, the pseudohermitian $\theta: = (i/2)(\bar{\partial} \rho - \partial\rho)$ is twice of the standard pseudohermitian structure on $\mathbb{S}^3$ and hence $\lambda_1(\mathbb{S}^3, \theta) = \tfrac{1}{2}$. Observe that $(\mathbb{S}^3,\theta)$ is CR immersed into $\mathbb{S}^5 \subset \mathbb{C}^{3}$ via H. Alexander's map $F(z,w) := (z^2, \sqrt{2}zw, w^2)$ and the Corollary~\ref{cor:1} applies. Thus, the constancy of $r[\rho]$ does \emph{not} implies that the equality occurs in \eqref{e:ub}.
\end{example}
\section{Webster scalar curvature and proofs of Theorem~\ref{thm:lb} and Corollary 1.4.}
In this section, we prove a formula for the Webster scalar curvature of the volume-normalized pseudohermitian structure of a real hypersurface.
\begin{proposition} \label{prop:wsc}
	Let $M$ be a strictly pseudoconvex hypersurface given by $\rho=0$ and $\theta$ a pseudohermitian structure given by \(\theta = (i/2) (\bar{\partial} \rho - \partial\rho)\). Then the Webster scalar curvature is given by 
	\begin{equation}
	R_{\theta}
	=
	n(n+1)r[\rho] - nN_{\rho} \log J[\rho] + \tfrac{1}{2}\Delta_b \log J[\rho],
	\end{equation}
	where $\Delta_b$ is the (positive) sublaplacian defined by $\theta$. In particular, if $J[\rho] = 1 + O(\rho^3)$, then
	\begin{equation}
	R_{\theta} = n(n+1)r[\rho].
	\end{equation}
\end{proposition}
\begin{proof}
	Observe that the Webster Ricci tensor has components,
	\begin{equation}\label{e:ricci}
	R_{\alpha\bbar}
	=
	-D^{\rho}_{\alpha\bbar} \log J[\rho] + (n+1) r[\rho]\, h_{\alpha\bbar}.
	\end{equation}
	On the other hand, since $\psi^{\bbar j } \rho_j = \xi^{\bbar}$, etc., we can compute
	\begin{align}
	\frac{(\psi^{\bbar\alpha} - \xi^{\bbar} \xi^{\alpha}) \rho_{\alpha}}{\rho_w}
	=
	\frac{\xi^{\bbar} - \psi^{\bbar w}\rho_w - \xi^{\bbar}(1-\rho_w \xi^{w})}{\rho_w}
	=
	\xi^{\bbar} \xi^{w} - \psi^{\bbar w},
	\end{align}
	Similarly,
	\begin{align}
	\frac{(\psi^{\bbar\alpha} - \xi^{\bbar}\xi^{\alpha}) \rho_{\bbar}}{\rho_{\wbar}}
	=
	\xi^{\alpha} \xi^{\wbar} - \psi^{\alpha \wbar},
	\end{align}
	and
	\begin{align}
	\frac{(\psi^{\bbar\alpha} - \xi^{\bbar} \xi^{\alpha}) \rho_{\bbar}\rho_{\alpha}}{|\rho_{\wbar}|^2}=
	\xi^{w} \xi^{\wbar} - \psi^{\wbar w}.
	\end{align}
	We obtain,
	\begin{equation}
	-h^{\bbar \alpha} D^{\rho}_{\alpha\bbar}
	=
	(\xi^{j} \xi^{\kbar} - \psi^{\kbar j}) \partial_{j}\partial_{\kbar} = \widetilde{\Delta}_{\rho} = \tfrac{1}{2}\Delta_b - n N_{\rho}.
	\end{equation}
	Therefore, by \eqref{e:ricci}
	\begin{align*}
	R_{\theta}
	= 
	R_{\alpha}{}^{\alpha}
	& = 
	-h^{\bbar\alpha}D^{\rho}_{\alpha\bbar} \log J[\rho] + n(n+1) r[\rho] \notag \\
	& = 
	n(n+1)r[\rho] - nN_{\rho} \log J[\rho] + \tfrac{1}{2} \Delta_b \log J[\rho]. \qedhere
	\end{align*}
\end{proof}
\begin{proof}[Proof of Theorem~\ref{thm:lb}]
	Let $\rho$ be any defining function satisfying $J[\rho] >0$. Let $\tilde{\rho}$ be the ``second approximation'':
	\begin{equation}
	\tilde{\rho} = J[\rho]^{-1/(n+1)} e^{-B(z)}\rho (z),
	\end{equation}
	where
	\begin{equation}
	B(z)
	:=
	\frac{1}{2n(n+1)} \trace (H(-\log(-\rho)))^{-1} H(\log J[\rho]).
	\end{equation}
	Then $J[\tilde{\rho}] = 1 + O(\tilde{\rho}^2)$ and $B(z) = 0$ on $M$ (see \cite{L16}). The unique volume-normalized structure is given by
	\begin{equation}
	\Theta: = (i/2)(\bar{\partial}\tilde{\rho}-\partial\tilde{\rho})
	\end{equation}
	Since $B(z) = 0$ on $M$, we have
	\begin{equation}
	{\Theta} = e^{u} \theta, \quad u = -\frac{\log J[\rho]}{n+1} .
	\end{equation}
	On the other hand, if $R_{\Theta}$ is the Webster scalar curvature of $\Theta$, then by Lee's formula \cite{L}
	\begin{equation}
	e^{u}R_{\Theta} = R+ (n+1)\Delta_b u - n(n+1) |\partial_b u|^2.
	\end{equation}
	By Proposition~\ref{prop:wsc},
	\begin{align}
	e^u 	R_{\Theta}
	& = 
	R - \Delta_b \log J[\rho] - \frac{n}{n+1} |\partial_b \log J[\rho]|^2 \notag \\
	& = 
	n(n+1)r[\rho] - nN_{\rho} \log J[\rho] - \tfrac{1}{2}\Delta_b \log J[\rho] - \frac{n}{n+1} |\partial_b \log J[\rho]|^2 \notag \\
	& = 
	D[\rho].
	\end{align}
	Finally, observe that $e^{u} = J[\rho]^{-1/(n+1)}$, hence
	\begin{equation*}
	R_{\Theta}
	=
	J[\rho]^{1/(n+1)} D[\rho]. \qedhere
	\end{equation*}
\end{proof}
\begin{proof}[Proof of Corollary 1.4]
	Let $\Theta$ be the unique volume-normalized pseudohermitian structure on $M$. It is well-known that $\Theta$ is pseudo-Einstein and so
	\begin{equation}
	R_{\alpha\bbar} = (R/n)h_{\alpha\bbar} \geq \min (R/n)\, h_{\alpha\bbar}.
	\end{equation}
	Therefore, by Chanillo-Chiu-Yang estimate \cite{CCY},
	\begin{equation}
	\lambda_1 \geq \min R/(n+1)
	=
	\frac{1}{n+1}J[\rho]^{1/(n+1)} D[\rho].
	\end{equation}
	By \cite{LSW}, if $n\geq 2$, the equality occurs if and only if $M$ is the sphere. 
	
	If $\rho$ is a second approximate solution, i.e., $J[\rho] = 1 + O(\rho^2)$, then 
	\begin{equation}
	R_{\Theta}
	=
	D[\rho]
	=
	n(n+1)\det H[\rho].
	\end{equation}
	Therefore, in this case $\lambda_1 \geq n\min_M \det H[\rho]$. The proof is complete.
\end{proof}
\begin{remark}\rm In $\mathbb{C}^2$, our formula \eqref{e:scal} for $R_{\Theta}$ is related to Hammond's in \cite{Hammond}. To see this, let $M\subset \mathbb{C}^2$ be
	a strictly pseudoconvex real hypersurface and $p\in M$. Suppose that near
	$p$, there is a holomorphic coordinates $(z,w)$ centered at the origin such that $\rho$ has an expansion of the form
	\begin{equation}
	\rho = -\Im w + |z|^2 + \kappa |z|^4 + \gamma z \zbar ^3 + \gamma z^3 \zbar + \cdots , 
	\end{equation}
	where $\kappa$ and $\gamma$ are real. This is the volume-preserving normal form of Hammond \cite{Hammond}.
	One can calculate $D[\rho](0)=4\kappa$ and $J[\rho](0) = \frac{1}{4}$, and therefore, the Webster scalar curvature at the origin is
	\begin{equation}
	R_{\Theta}(0) = J[\rho]^{1/3} D[\rho](0) = 2\sqrt[3]{2}\, \kappa.
	\end{equation}
	This agrees with Hammond's result, except that the constant is different from the one in \cite{Hammond} due to a different normalization. 
\end{remark}

\end{document}